\documentclass[a4paper,12pt,leqno]{amsart}
\usepackage{amsmath}
\usepackage{amssymb}
\usepackage{amsthm}
\usepackage[left=1.1in, right=1.1in, top=.70in, bottom=.70in]{geometry}
\usepackage[colorlinks=true]{hyperref}

\theoremstyle{plain}
\newtheorem{thm}{Theorem}[section]

\newtheorem{lem}[thm]{Lemma}

\newtheorem{con}[thm]{Conjecture}

\numberwithin{equation}{section}
\title[Polynomial image coprime with the $n$th term of a linear recurrence]
      {On numbers $n$ with polynomial image coprime with the $n$th term of a linear recurrence}

\author[D.~Mastrostefano]{Daniele Mastrostefano}
\address{University of Warwick, Mathematics Institute, Zeeman Building, Coventry, CV4 7AL, England}
\email{danymastro93@hotmail.it}

\author[C.~Sanna]{Carlo Sanna\textsuperscript{$\dagger$}}
\address{Universit\`a degli Studi di Torino\\Department of Mathematics\\Turin, Italy}
\thanks{$\dagger$ C.~Sanna is a member of INdAM group GNSAGA}
\email{carlo.sanna.dev@gmail.com}


\begin{document}

\begin{abstract}
Let $F$ be an integral linear recurrence, $G$ be an integer-valued polynomial splitting over the rationals, and $h$ be a positive integer. 
Also, let $\mathcal{A}_{F,G,h}$ be the set of all natural numbers $n$ such that $\gcd(F(n), G(n)) = h$.
We prove that $\mathcal{A}_{F,G,h}$ has a natural density.
Moreover, assuming $F$ is non-degenerate and $G$ has no fixed divisors, we show that $\mathbf{d}(\mathcal{A}_{F,G,1}) = 0$ if and only if $\mathcal{A}_{F,G,1}$ is finite.
\end{abstract}

\maketitle
\section{Introduction}

An \emph{integral linear recurrence} is a sequence of integers $F(n)_{n\geq 0}$ such that
\begin{equation}\label{equ:linear}
F(n)=a_{1}F(n-1)+\cdots+a_{k}F(n-k),
\end{equation}
for all integers $n\geq k$, for some fixed $a_1, \dots, a_k\in\mathbb{Z}$, with $a_k\neq 0$,
We recall that $F$ is said to be \emph{non-degenerate} if none of the ratios $\alpha_{i}/\alpha_{j}$ $(i \neq j)$ is a root of unity, where $\alpha_{1}, \dots,\alpha_{r}\in\mathbb{C}^*$ are all the pairwise distinct roots of the \emph{characteristic polynomial}
\begin{equation*}
\psi_{F}(X)=X^{k}-a_{1}X^{k-1}-\dots-a_{k}.
\end{equation*}
Moreover, $F$ is said to be a \emph{Lucas sequence} if $F(0)=0, F(1)=1,$ and $k=2$.
In~particular, the Lucas sequence with $a_1=a_2=1$ is known as the \emph{Fibonacci sequence}.
We refer the reader to \cite[Ch.~1--8]{MR1990179} for the basic terminology and theory of linear recurrences.

Given two integral linear recurrences $F$ and $G$, the arithmetic relations between the corresponding terms $F(n)$ and $G(n)$ have interested many authors.
For instance, the study of the positive integers $n$ such that $G(n)$ divides $F(n)$ is a classic problem which goes back to Pisot, and the major results have been given by van~der~Poorten~\cite{MR929097}, Corvaja and Zannier~\cite{MR1692189, MR1918678}. (See also~\cite{MR3679793} for a proof of the last remark in \cite{MR1918678}.)
In~particular, the special case in which $G = I$, where $I$ is the identity sequence given by $I(n) = n$ for all integers $n$, has attracted much attention; with results given by Alba~Gonz{\'a}lez, Luca, Pomerance, and Shparlinski~\cite{MR2928495}, under the hypothesis that $F$ is simple and non-degenerate, and by Andr\'e-Jeannin~\cite{MR1131414}, Luca and Tron~\cite{MR3409327}, Sanna~\cite{MR3606950}, Smyth~\cite{MR2592551}, and Somer~\cite{MR1271392}, when $F$ is a Lucas sequence or the Fibonacci sequence.

Furthermore, for large classes of integral linear recurrences $F,G$, upper bounds for $\gcd(F(n), G(n))$ have been proved by Bugeaud, Corvaja, and Zannier~\cite{MR1953049}, and by Fuchs~\cite{MR1964201}.
Also, Leonetti and Sanna~\cite{LS001} studied the integers of the form $\gcd(F(n), n)$, when $F$ is the Fibonacci sequence; while Sanna~\cite{SannaLog} determined all the moments of the function $n \mapsto \log(\gcd(F(n), n))$, for any non-degenerate Lucas sequence $F$.

For two integral linear recurrences $F,G$ and a positive integer $h$, let us define
\begin{equation*}
\mathcal{A}_{F,G,h} := \big\{ n \in \mathbb{N} : \gcd(F(n), G(n)) = h \big\} ,
\end{equation*}
and put also $\mathcal{A}_{F,G} := \mathcal{A}_{F,G,1}$.
Sanna~\cite{San001} proved the following result on $\mathcal{A}_{F,I}$.

\begin{thm}
Let $F$ be a non-degenerate integral linear recurrence.
Then the set $\mathcal{A}_{F, I}$ has a natural density.
Moreover, if $F / I$ is not a linear recurrence (of rational numbers) then $\mathbf{d}(\mathcal{A}_{F,I}) > 0$.
Otherwise, $\mathcal{A}_{F, I}$ is finite and, a fortiori, $\mathbf{d}(\mathcal{A}_{F, I}) = 0$.
\end{thm}

In the special case of the Fibonacci sequence, Sanna and Tron~\cite{ST001} gave a more precise result:

\begin{thm}
Assume $F$ is the Fibonacci sequence.
Then, for each positive integer $h$, the natural density of $\mathcal{A}_{F,I,h}$ exists and is given by
\begin{equation*}
\mathbf{d}(\mathcal{A}_{F,I,h}) = \sum_{d = 1}^\infty \frac{\mu(d)}{\operatorname{lcm}(dh, z(dh))} ,
\end{equation*}
where $\mu$ is the M\"obius function and $z(m)$ denotes the least positive integer $n$ such that $m$ divides $F(n)$.
Moreover, $\mathbf{d}(\mathcal{A}_{F,I,h}) > 0$ if and only if $\mathcal{A}_{F,I,h} \neq \varnothing$ if and only if $h = \gcd(\ell, F_\ell)$ with $\ell := \operatorname{lcm}(h, z(h))$.
\end{thm}

Also, they pointed out that their result can be extended to any non-degenerate Lucas sequence $F$ with $\gcd(a_1, a_2) = 1$; while Kim~\cite{Kim} gave an analog result for elliptic divisibility sequences.

Trying to extend the previous result to $\mathcal{A}_{F,G,h}$ for two arbitrary integral linear recurrences is quite tempting.
However, already establishing if the set $\mathcal{A}_{F,G}$ is infinite seems too difficult for the current methods.
Indeed, the following conjecture of Ailon and Rudnick~\cite{MR2046966} is open.

\begin{con}
Let $a, b$ be two multiplicatively independent non-zero integers with $\gcd(a - 1, b - 1) = 1$.
Then, for the linear recurrences $F(n) = a^n - 1$ and $G(n) = b^n - 1$, the set $\mathcal{A}_{F, G}$ is infinite.
\end{con}

In this paper, we focus on the special case in which the linear recurrence $G$ is an integer-valued polynomial splitting over the rationals.
Our main result is the following:

\begin{thm}\label{thm:main}
Let $F$ be an integral linear recurrence, $G$ be an integer-valued polynomial with all roots in $\mathbb{Q}$, and $h$ be a positive integer.
Then, the set $\mathcal{A}_{F,G,h}$ has a natural density.
Moreover, if $F$ is non-degenerate and $G$ has no fixed divisors (and $h = 1$), then $\mathbf{d}(\mathcal{A}_{F,G}) = 0$ if and only if $\mathcal{A}_{F,G}$ is finite.
\end{thm}

It would be interesting to prove Theorem~\ref{thm:main} for any integer-valued polynomial $G$, dropping the hypothesis that all the roots of $G$ must be rational or eliminating the presence of a fixed divisor.
However, doing so presents some difficulties, which we will highlight in the last section.

\subsection*{Notation}

Throughout, the letter $p$ will always denote a prime number, and we write $\nu_p$ for the $p$-adic valuation.
For a set of positive integers $\mathcal{S}$, we put $\mathcal{S}(x):=\mathcal{S}\cap [1,x]$ for all $x\ge 1$, and we recall that the natural density $\mathbf{d}(\mathcal{S})$ of $\mathcal{S}$ is defined as the limit of the ratio $\#\mathcal{S}(x) / x$ as $x \to +\infty$, whenever this exists.
We employ the Landau--Bachmann ``Big Oh'' and ``little oh'' notations $O$ and $o$, as well as the associated Vinogradov symbols $\ll$ and $\gg$, with their usual meanings. 
Any dependence of the implied constants is explicitly stated or indicated with subscripts.

\section{Preliminary results}

In this section, we collect some definitions and preliminary results needed in the later proofs.
Let $F$ be a non-degenerate integral linear recurrence satisfying \eqref{equ:linear} and let $\psi_F$ be its characteristic polynomial.
To avoid trivialities, we assume that $F$ is not identically zero.
Moreover, let $\mathbb{K}$ be the splitting field of $\psi_{F}$ over $\mathbb{Q}$, and let $\alpha_1, \ldots, \alpha_r \in \mathbb{K}$ be all the distinct roots of $\psi_F$.

It is well known that there exist non-zero polynomials $f_1, \ldots, f_r \in \mathbb{K}[X]$ such that 
\begin{equation}\label{equ:genpowsum}
F(n) = \sum_{i = 1}^r f_i(n)\,\alpha_i^n ,
\end{equation}
for all integers $n \geq 0$.
In fact, the expression \eqref{equ:genpowsum} is known as the \emph{generalized power sum} representation of $F$ and is unique (assuming the roots $\alpha_1, \ldots, \alpha_r$ are distinct, and up to the order of the addends).

Let $G$ be an integer-valued polynomial, and let $h$ be a positive integer.
We begin with two basic lemmas about $\mathcal{A}_{F,G,h}$.

\begin{lem}\label{lem:disjointunion}
We have that $\mathcal{A}_{F,G,h}$ is the disjoint union of a finite set and finitely many sets of the form $a\mathcal{A}_{\widetilde{F}, \widetilde{G}} + b$, where $a,b$ are positive integers, $\widetilde{F}$ is a non-degenerate integral linear recurrence, and $\widetilde{G}$ is an integer-valued polynomial.
\end{lem}
\begin{proof}
First, it is well known and easy to prove that there exists a positive integer $c$ such that, setting $F_j(m) := F(cm + j)$ for all non-negative integers $m$ and $j < c$, each $F_j$ is an integral linear recurrence which is non-degenerate or identically zero.
Then, $\mathcal{A}_{F,G,h}$ is the disjoint union of the sets $\mathcal{A}_{F_j, G_j,h}$, where $G_j(m) := G_j(cm + j)$.
Thus, without loss of generality, we can assume that $F$ is non-degenerate.

Clearly, if $n \in \mathcal{A}_{F,G,h}$ then $h$ divides both $F(n)$ and $G(n)$.
Since integral linear recurrences (and, in particular, integer-valued polynomials) are definitively periodic modulo any positive integer, there exist a finite set $\mathcal{E}$ and positive integers $a, b_1, \ldots, b_t$ such that $h \mid \gcd(F(n),G(n))$ if and only if $n \in \mathcal{E}$ or $n = a m + b_i$, for some positive integer $m$ and some $i \in \{1,\dots,t\}$.
Moreover, if $n = am + b_i$, for some integers $m \geq 1$ and $i \in \{1,\dots,t\}$, then $n \in \mathcal{A}_{F,G,h}$ if and only if $m \in \mathcal{A}_{\widetilde{F}_j, \widetilde{G}_j}$, where $\widetilde{F}_i(\ell) := F(a\ell + b_i) / h$ and $\widetilde{G}_i(\ell) := G(a\ell + b_i) / h$ for all integers $\ell \geq 0$.
In~particular, $\widetilde{F}_i$ is a non-degenerate integral linear recurrence and $\widetilde{G}_i$ is an integer-valued polynomial.
So we have proved that $\mathcal{A}_{F,G,h}$ is the disjoint union of the finite set $\mathcal{E}$ and $a \mathcal{A}_{\widetilde{F}_i, \widetilde{G}_i} + b_i$, for $i=1,\dots,t$, as desired.
\end{proof}

\begin{lem}\label{lem:commonfactor}
If $G, f_1, \ldots, f_r$ have a non-trivial common factor, then $\mathcal{A}_{F,G}$ is finite.
\end{lem}
\begin{proof}
Suppose $X - \beta$ divides each of $G, f_1, \ldots, f_r$, for some algebraic number $\beta$.
Let $g \in \mathbb{Q}[X]$ be the minimal polynomial of $\beta$ over $\mathbb{Q}$.
Clearly, $g$ divides $G$.
Also, if $\mathbb{L}$ is the splitting field of $gGf_1\cdots f_r$, then for each $\sigma \in \operatorname{Gal}(\mathbb{L} / \mathbb{Q})$ we have
\begin{equation*}
F(n) = \sigma(F(n)) = \sum_{i = 1}^r (\sigma f_i)(n) (\sigma(\alpha_i))^n ,
\end{equation*}
for all positive integers $n$.
In particular, $\sigma(\beta)$ is a root of each $\sigma f_i$, since $\beta$ is a root of each $f_i$.  
Therefore, by the uniqueness of the generalized power sum expression of a linear recurrence, we get that $\sigma(\beta)$ is a root of each $f_i$ and, as a consequence, $g$ divides each $f_i$.
Now let $B$ be a positive integer such that all the polynomials $BG / g, Bf_1 / g, \ldots, Bf_r / g$ have coefficients which are algebraic integers.
Then, it follows easily that $B F(n) / g(n)$ and $B G(n) / g(n)$ are both integers, for all positive integers $n$. (Note that $g(n) \neq 0$ since $g$ is irreducible in $\mathbb{Q}[X]$.)
Hence, $n \in \mathcal{A}_{F,G}$ implies $g(n) \mid B$, which is possible only for finitely many positive integers $n$.
\end{proof}

If $r \geq 2$, then for all integers $x_1, \ldots, x_r$ we set
\begin{equation*}
D_F(x_1, \ldots, x_r) := \det(\alpha_i^{x_j})_{1 \leq i, j \leq r} ,
\end{equation*}
and for any prime number $p$ not dividing $a_k$ we define $T_F(p)$ as the greatest integer $T \geq 0$ such that $p$ does not divide
\begin{equation*}
\prod_{1 \leq x_2, \ldots, x_r \leq T} \max\{1, |N_{\mathbb{K}} (D_F(0, x_2, \ldots, x_r))|\} ,
\end{equation*}
where the empty product is equal to $1$, and $N_\mathbb{K}(\alpha)$ denotes the norm of $\alpha \in \mathbb{K}$ over $\mathbb{Q}$.
It is known that such $T$ exists~\cite[p.~88]{MR1990179}. 
If $r = 1$, then we set $T_F(p) := +\infty$ for all prime numbers $p$ not dividing $a_1$. 

Finally, for all $\gamma > 0$, we define
\begin{equation*}
\mathcal{P}_{F,\gamma} := \{p : p \nmid a_k, \; T_F(p) < p^\gamma \} .
\end{equation*}
The next lemma shows that $T_F(p)$ is usually larger than a power of $p$.

\begin{lem}\label{lem:PFgamma}
For all $\gamma \in {]0,1/r]}$ and $x \geq 2^{1/ \gamma}$, we have
\begin{equation*}
\#\mathcal{P}_{F,\gamma}(x) \ll_F \frac{x^{r\gamma}}{\gamma \log x} .
\end{equation*}
\end{lem}
\begin{proof}
See \cite[Lemma~2.1]{MR2928495}.
\end{proof}

From the previous estimate is easy to deduce the following bound.

\begin{lem}\label{lem:pTFp}
We have
\begin{equation*}
\sum_{p > y} \frac1{p T_F(p)} \ll_{F} \frac1{y^{1/(r+1)}} ,
\end{equation*}
for all sufficiently large $y$.
\end{lem}
\begin{proof}
We split the series into two parts, separating between prime numbers which belong to $\mathcal{P}_{F,\gamma}$ and which do not.
In the first case, by partial summation and Lemma~\ref{lem:PFgamma}, for a fixed $\gamma \in {]0,1/r[}$, we find
\begin{equation}\label{equ:2.4}
\sum_{\substack{p > y \\ p \in \mathcal{P}_{F,\gamma}}} \frac1{pT_F(p)} \leq \sum_{\substack{p > y \\ p \in \mathcal{P}_{F,\gamma}}} \frac1{p} = \left[\frac{\#\mathcal{P}_{F,\gamma}(t)}{t}\right]_{t = y}^{+\infty} + \int_y^{+\infty} \frac{\#\mathcal{P}_{F,\gamma}(t)}{t^2}\mathrm{d}t \ll_{F,\gamma} \frac1{y^{1 - r\gamma}}.
\end{equation}
On the other hand, in the second case we get
\begin{equation}
\label{equ:2.5}
\sum_{\substack{p > y \\ p \notin \mathcal{P}_{F,\gamma}}} \frac1{pT_F(p)} \leq \sum_{\substack{p > y }} \frac1{p^{1 + \gamma}} \ll \int_y^{+\infty} \frac{\mathrm{d}t}{t^{1 + \gamma}} \ll_\gamma \frac1{y^{\gamma}}.
\end{equation}
If we put $\gamma:=1/(r+1)$ and collect together the estimates \eqref{equ:2.4} and \eqref{equ:2.5} we obtain the thesis.
\end{proof}

The next lemma is an upper bound in terms of $T_F(p)$ for the number of solutions of a certain congruence modulo $p$ involving $F$.
The proof proceeds essentially like the one of \cite[Lemma~2.2]{MR2928495}, which in turn relies on previous arguments given in~\cite{MR917803} (see also~\cite{Shp87} and \cite[Theorem~5.11]{MR1990179}).
We include it for completeness.

\begin{lem}\label{lem:Fpml}
Let $p$ be a prime number dividing neither $a_k$ nor the denominator of any of the coefficients of $f_1, \ldots, f_r$.
Moreover, let $\ell \geq 0$ be an integer such that $f_1(\ell), \ldots, f_r(\ell)$ are not all zero modulo some prime ideal of $\mathcal{O}_\mathbb{K}$ lying over $p$.
Then, for all $x > 0$, the number of integers $m \in [0, x]$ such that $F(pm + \ell) \equiv 0 \pmod p$ is
\begin{equation*}
O_r\!\left(\frac{x}{T_F(p)} + 1\right) .
\end{equation*}
\end{lem}
\begin{proof}
For $r = 1$ the claim can be proved quickly using \eqref{equ:genpowsum}.
Hence, assume $r \geq 2$.
Let $\mathcal{I}$ be an interval of $T_F(p)$ consecutive non-negative integers, and let $m_1 < \cdots < m_s$ be all the integers $m \in \mathcal{I}$ such that $F(pm + \ell) \equiv 0 \pmod p$.
Also, let $\pi$ be a prime ideal of $\mathcal{O}_\mathbb{K}$ lying over $p$.
Then, by \eqref{equ:genpowsum}, and since no denominator of the coefficients of $f_1, \dots, f_r$ belongs to $\pi$, we have
\begin{equation}\label{equ:Fpmell0}
\sum_{i = 1}^r f_i(\ell)(\alpha_i)^{\ell+pm_1} \,(\alpha_i^{p})^{m_j - m_1} \equiv \sum_{i = 1}^r f_i(pm_j + \ell)\,\alpha_i^{pm_j + \ell} \equiv 0 \pmod\pi ,
\end{equation}
for $j = 1, \dots, s$.
By a result of Schlickewei~\cite{MR1397562}, there exists a constant $C(r)$, depending only on $r$, such that for any $B_1, \dots, B_r \in \mathbb{K}$, not all zero, the exponential equation
\begin{equation*}
\sum_{i=1}^r B_i \alpha_i^x = 0
\end{equation*}
has at most $C(r)$ solutions in positive integers $x$.
Suppose $s \geq C(r) + r$.
Put $x_1 := 0$ and, setting $\mathcal{X}_2 := \{m_j - m_1 : j=2,\dots,s\}$, pick some $x_2 \in \mathcal{X}_2$ such that
\begin{equation*}
\det(\alpha_i^{x_j})_{1 \leq i,j \leq 2} \neq 0 .
\end{equation*}
This is possible by the mentioned result of Schlickewei, since 
\begin{equation*}
\#\mathcal{X}_2 = s - 1 \geq C(r) + r - 1 > C(r) .
\end{equation*}
For $r \geq 3$, set $\mathcal{X}_3 := \mathcal{X}_2 \setminus \{x_2\}$ and pick $x_3 \in \mathcal{X}_3$ such that
\begin{equation}\label{equ:nontrivialdet}
\det(\alpha_i^{x_j})_{1 \leq i,j \leq 3} \neq 0 .
\end{equation}
Again, this is still possible since, by the choice of $x_2$, \eqref{equ:nontrivialdet} is a non-trivial exponential equation and 
\begin{equation*}
\#\mathcal{X}_3 = s - 2 \geq C(r) + r - 2 > C(r) .
\end{equation*}
Continuing this way, after $r - 1$ steps, we obtain integers $x_2, \dots, x_r \in {[1, T_F(p)[}$ such that
\begin{equation}\label{equ:Dneq0}
D_F(0, x_2, \ldots, x_r) \neq 0 .
\end{equation}
Now, since $f_i(\ell)$ are not all zero modulo $\pi$, by \eqref{equ:Fpmell0} we get
\begin{equation*}
\det(\alpha_i^{px_j})_{1 \leq i,j \leq r} \equiv 0 \pmod \pi ,
\end{equation*}
so that
\begin{equation*}
N_\mathbb{K}(D_F(0, x_2, \ldots, x_r))^p = N_\mathbb{K}(\det(\alpha_i^{x_j}))^p \equiv N_\mathbb{K}(\det(\alpha_i^{px_j})) \equiv 0 \pmod p ,
\end{equation*}
which is impossible by the definition of $T_F(p)$ and condition \eqref{equ:Dneq0}.
Hence, $s < C(r) + r$ and the desired claim follows easily.
\end{proof}

We conclude this section with the next lemma.

\begin{lem}\label{lem:fiellnonzero}
If $\gcd(G, f_1, \dots, f_r) = 1$ then there are only finitely many prime numbers $p$ such that $p \mid G(\ell)$, for some integer $\ell$, and $f_1(\ell), \dots, f_r(\ell)$ are all zero modulo some prime ideal of $\mathcal{O}_\mathbb{K}$ lying over $p$.
\end{lem}
\begin{proof}
By B\'{e}zout's theorem for polynomials in $\mathbb{K}[X]$, there exist $h_0,\dots, h_r \in \mathbb{K}[X]$ such that
\begin{equation*}
G h_0+f_1 h_1+\cdots+f_r h_r = 1.
\end{equation*}
Let $B$ be a positive integer such that all the coefficients of $Bh_0,\dots,Bh_r$ are algebraic integers.
If $\pi$ is a prime ideal of $\mathcal{O}_\mathbb{K}$ lying over $p$ such that $f_1(\ell), \dots, f_r(\ell)$ are all zero modulo $\pi$, then
\begin{equation*}
B \equiv B G(l) h_0(l)+Bf_1(l) h_1(l)+\cdots+ Bf_r(l)h_r(l) \equiv 0 \pmod \pi ,
\end{equation*}
since $p \mid G(\ell)$.
Hence, $p \mid B$ and this is possible only for finitely many primes $p$.
\end{proof}

\section{Proof of Theorem~\ref{thm:main}}

We begin by proving that $\mathcal{A}_{F,G,h}$ has a natural density.
First, in light of Lemma~\ref{lem:disjointunion}, without loss of generality, we can assume that $F$ is non-degenerate and not identically zero, and that $h = 1$.
By Lemma~\ref{lem:commonfactor}, if $G, f_1, \dots, f_r$ share a non-trivial common factor then $\mathcal{A}_{F,G}$ is finite and, obviously, $\mathbf{d}(\mathcal{A}_{F,G}) = 0$.
Hence, assume $\gcd(G, f_1, \dots, f_r) = 1$.

Put $\mathcal{C}_{F,G} := \mathbb{N} \setminus \mathcal{A}_{F,G}$ so that, equivalently, we have to prove that the natural density of $\mathcal{C}_{F,G}$ exists. 
For each $y > 0$, we split $\mathcal{C}_{F,G}$ into two subsets:
\begin{align*}
\mathcal{C}_{F,G,y}^- &:= \big\{n \in \mathcal{C}_{F,G} : p \mid \gcd(G(n), F(n)) \text{ for some } p \leq y\big\} , \\
\mathcal{C}_{F,G,y}^+ &:= \mathcal{C}_{F,G} \setminus \mathcal{C}_{F,G,y}^- .
\end{align*}
Recalling that $F,G$ are definitively periodic modulo $p$, for any prime number $p$, we see that $\mathcal{C}_{F,G,y}^-$ is a union of finitely many arithmetic progressions and a finite subset of $\mathbb{N}$. 
In particular, $\mathcal{C}_{F,G,y}^-$ has a natural density.
If we put $\delta_y := \mathbf{d}(\mathcal{C}_{F,G,y}^-)$, then it is clear that $\delta_y$ is a bounded non-decreasing function of $y$.
Hence, the limit 
\begin{equation}\label{equ:3.1}
\delta := \lim_{y \to +\infty} \delta_y
\end{equation}
exists finite.
We shall prove that $\mathcal{C}_{F,G}$ has natural density $\delta$. 
If $n \in \mathcal{C}_{F,G,y}^+(x)$ then there exists a prime $p > y$ such that $p \mid G(n)$ and $p \mid F(n)$.
In particular, we can write $n = pm + \ell$, for some non-negative integers $m \leq x/p$ and $\ell < p$, with $p \mid G(\ell)$.
For sufficiently large $y$, how large depending only on $F,G$, we have that $p$ divide neither $a_k$ nor any of the denominators of the coefficients of $f_1,\dots , f_r$, and that, by Lemma~\ref{lem:fiellnonzero}, the terms $f_1(\ell), \dots, f_2(\ell)$ are not all zero modulo some prime ideal of $\mathcal{O}_\mathbb{K}$ lying over $p$.
On the one hand, by Lemma \ref{lem:Fpml}, the number of possible values for $m$ is 
\begin{equation*}
O_r\!\left(\frac{x}{pT_F(p)}+1\right) .
\end{equation*}
On the other hand,
for sufficiently large $y$, depending on $G$, the number of possible values for $\ell$ is at most $\deg(G)$.
Furthermore, we have $p \ll_G x$, since all the roots of $G$ are in $\mathbb{Q}$.
(Note that this property is preserved by the reduction to $\widetilde{G}$ in Lemma~\ref{lem:disjointunion}.)
Therefore, setting $\gamma := 1/(r+1)$, we get
\begin{equation}
\label{eq: 3.2}
\#\mathcal{C}_{F,G,y}^+(x) \ll_{F,G} \sum_{y < p \ll_G x} \left(\frac{x}{pT_F(p)} + 1\right) \ll_{F,G} \frac{x}{y^\gamma} + \frac{x}{\log x} ,
\end{equation}
where we used Lemma~\ref{lem:pTFp} and Chebyshev's estimate for the number of primes not exceeding $x$.
Thus, we obtain that 
\begin{align}\label{equ:3.3}
\limsup_{x \to +\infty} \left|\frac{\#\mathcal{C}_{F,G}(x)}{x} - \delta_y\right| &= \limsup_{x \to +\infty} \left|\frac{\#\mathcal{C}_{F,G}(x)}{x} - \frac{\#\mathcal{C}_{F,G,y}^-(x)}{x}\right| \\
 &= \limsup_{x \to +\infty} \frac{\#\mathcal{C}_{F,G,y}^+(x)}{x} \ll_{F,G} \frac1{y^{\gamma}} . \nonumber
\end{align}
Hence, letting $y \to +\infty$ in \eqref{equ:3.3} and using \eqref{equ:3.1}, we get that $\mathcal{C}_{F,G}$ has natural density~$\delta$.

At this point, assuming that $G$ has no fixed divisors, it remains only to prove that the natural density of $A_{F,G}$ is positive.
In turn, this is equivalent to $\delta < 1$.
Clearly,
\begin{equation*}
\mathcal{C}_{F,G,y}^- \subseteq \big\{n \in \mathbb{N} : p \mid G(n) \text{ for some } p \leq y\big\} .
\end{equation*}
Hence, by standard sieving arguments (see, e.g.,~\cite[\S1.2.3, Eq.~3.3]{MR1836967}), we have
\begin{equation*}
\frac{\#\mathcal{C}_{F,G,y}^-(x)}{x} \leq 1 - \prod_{p \leq y}\left(1 - \frac{\rho_G(p)}{p}\right) + O_G\!\left(\frac1{x}\sum_{d \mid P(y)}\rho_G(d)\right),
\end{equation*}
where $P(y):=\prod_{p\leq y}p$, while $\rho_G$ is the completely multiplicative function supported on squarefree numbers and satisfying
\begin{equation*}
\rho_G(p) := \frac{\#\left\{z \in \{1, \ldots, p^{1 + \nu_p(B)}\} : BG(z) \equiv 0 \!\!\!\!\pmod {p^{1 + \nu_p(B)}}\right\}}{p^{\nu_p(B)}} ,
\end{equation*}
for all prime numbers $p$, where $B$ is a positive integer such that $BG \in \mathbb{Z}[X]$.
Since $G$ has no fixed divisors, we have $\rho_G(p) < p$ for all prime numbers $p$.
Also, $\rho_G(p) \leq \deg(G)$ for all sufficiently large prime numbers $p$.
Therefore,
\begin{equation*}
\prod_{p \leq y}\left(1 - \frac{\rho_G(p)}{p}\right) \gg_G \frac1{(\log y)^{\deg(G)}} ,
\end{equation*}
if $y$ is large enough, which implies that 
\begin{equation}
\label{equ:3.4}
\limsup_{x \to +\infty} \frac{\#\mathcal{C}_{F,G,y}^-(x)}{x} \leq 1 - \frac{c_1}{(\log y)^{\deg(G)}},
\end{equation}
where $c_1 > 0$ is a constant depending on $G$.
Hence, putting together \eqref{equ:3.3} and \eqref{equ:3.4}, we get
\begin{align}
\label{equ:3.5}
\delta &= \lim_{x \to +\infty} \frac{\#\mathcal{C}_{F,G}(x)}{x} \leq \limsup_{x \to +\infty} \frac{\#\mathcal{C}_{F,G,y}^-(x)}{x} + \limsup_{x \to +\infty} \frac{\#\mathcal{C}_{F,G,y}^+(x)}{x} \\
&\leq 1 - \left(\frac{c_1}{(\log y)^{\deg(G)}}-\frac{c_2}{y^{\gamma}}\right) , \nonumber
\end{align}
where $c_2 > 0$ is a constant depending on $F,G$.
Finally, picking a sufficiently large $y$, depending on $c_1$ and $c_2$, the bound \eqref{equ:3.5} yields $\delta < 1$, as desired.
The proof of Theorem~\ref{thm:main} is complete.

\section{Concluding remarks}

\subsection{The case in which $G$ has a fixed divisor}
Suppose that $F$ is a non-degenerate integral linear recurrence and that $G$ is an integer-valued polynomial with all roots in $\mathbb{Q}$ and having a fixed divisor $d > 1$.
In order to study $\mathcal{A}_{F,G}$, one could try to reduce from this general situation to the one where there is no fixed divisor, so that Theorem~\ref{thm:main} can be applied.
However, the strategy used in Lemma~\ref{lem:disjointunion}, that is, writing $\mathcal{A}_{F,G}$ as the disjoint union of a finite set and finitely many sets of the form $a\mathcal{A}_{\widetilde{F}, \widetilde{G}} + b$, this time does not work.
The issue here is that the resulting polynomials $\widetilde{G}$ may have fixed divisors. 
For example, let $F$ be the Fibonacci sequence and $G(n) = n(n + 1)$, so that $d = 2$.
Then, $2 \nmid F(n)$ if and only if $n \equiv 1,2 \pmod 3$, so that $\mathcal{A}_{F,G}$ is the disjoint union of $\mathcal{A}_{\widetilde{F}_1,\widetilde{G}_1}$ and $\mathcal{A}_{\widetilde{F}_2,\widetilde{G}_2}$, where $\widetilde{F}_i(m) = F(3m + i)$ and $\widetilde{G}_i(m) = G(3m + i) / 2$, for $i=1,2$.
Now, $G_1(m) = (9m^2 + 9m + 2) / 2$ has no fixed divisors, but $G_2(m) = (9m^2 + 15m + 6) / 2$ gained $3$ as a new fixed divisor.

\subsection{The case in which $G$ does not split over the rationals}

We note that there are examples of integral linear recurrences $F$ and integer-valued polynomials $G$, not splitting over the rationals, such that $\mathcal{A}_{F,G}$ has a positive density for elementary reasons.
For instance, for the following couple
\begin{equation*}
F(n) = (n^{2} + 1)5^{n} + (n^{2} + 2)3^{n}, \quad G(n) = (n^{2}+1)(n^{2}+2) ,
\end{equation*}
we have $\mathcal{A}_{F,G} = \mathbb{N}$.
Indeed, suppose by contradiction that there exists a prime $p$ dividing both $F(n)$ and $G(n)$.
Then, $p\mid (n^{2}+1)$ or $p\mid (n^{2}+2)$, exclusively.
In the first case, since $p\mid F(n)$, we get $p\mid 3^{n}$, that is, $p = 3$, which is not possible, since $n^2 + 1$ is never a multiple of $3$.
The second case is similar.

However, except for those easy situations, we think that if $G$ does not split over the rationals, then the study of $\mathcal{A}_{F,G}$ requires different methods than those employed in this paper.
In fact, if $p \mid G(n)$ we can only say that $p \ll_G x^{\deg(G)}$ and, for $\deg(G) \geq 2$, this does not allow one to conclude that $\limsup_{x \to +\infty} \mathcal{C}_{F,G,y}^{+}(x) / x = o((\log y)^{-\deg(G)})$, as $y \to +\infty$, which is a key step in the proof of Theorem~\ref{thm:main}.
Actually, in the following we provide a heuristic for the claim that $\mathcal{C}_{F,G,y}^{+}(x) \gg x$ for all $y$.
First, we can split $\mathcal{C}_{F,G,y}^{+}(x)$ into two parts: 
the first one is 
\begin{equation*}
\{n\leq x : \gcd(F(n), G(n))\neq 1\ \mathrm{and}\ p\mid\gcd(F(n), G(n)) \Rightarrow y<p\leq x \},
\end{equation*}
which can be handled as in \eqref{eq: 3.2}, whereas the second one is 
\begin{equation}\label{equ:heuristic}
\{n\leq x: \exists p\mid\gcd(F(n), G(n))\ \mathrm{with}\ p>x\} ,
\end{equation}
which, by our heuristic, we believe it should have a cardinality $\gg x$. 

For the sake of simplicity, we consider only the case where $F$ is the Fibonacci sequence and $G(n)=n^{2}+1$. 
By a result of Everest and Harman about the existence of primitive divisors of quadratic polynomials \cite[Theorem~1.4]{MR2428520}, we have 
\begin{equation*}
\#\!\left\{n \leq x : \exists p > x \text{ with } p \mid G(n)\right\} \gg x ,
\end{equation*}
so that
\begin{equation*}
\mathbb{P}_{x}\!\left[\exists p > x \text{ with } p \mid G(n) \right] \gg 1,
\end{equation*}
where we consider the events in the probability space $([x], \mathcal{P}[x], \mathbb{P}_{x})$, with $[x]=\{n\leq x\}$ and $\mathbb{P}_{x}$ is the discrete uniform measure on $[x]$.
Let $z_{F}(m)$ be the least positive integer $n$ such that $m\mid F(n)$. 
It is well known that $p \mid F(n)$ if and only if $z_F(p)\mid n$.
This means that $\mathbb{P}_x\!\left[p \mid F(n)\right]$ is roughly $1/z_F(p)$.
Therefore, interpreting the events of being divisible by different prime numbers as independent, we expect that
\begin{align*}
&\mathbb{P}_{x}\!\left[\exists p > x \text{ with } p \mid F(n) \right] \geq 1- \mathbb{P}_x\!\left[p \nmid F(n) \text{ for all } p \text{ with }x < p \ll x^2\right] \\
=1- &\prod_{ p:\ x<p\ll x^2 }\left(1-\frac1{z_F(p)}\right) > 1-\prod_{ p:\ x<p\ll x^2 }\left(1-\frac1{p+1}\right) > 1/2 + o(1),
\end{align*}
as $x \to +\infty$, since $z_F(p)\leq p+1$ and thanks to Mertens' Theorem. 
Assuming independence between the events that a prime divides $F(n)$ or $G(n)$, we deduce that the expected value of the cardinality of \eqref{equ:heuristic} is
\begin{align*}
&\sum_{n \leq x}\mathbb{P}_{x}\!\left[\exists p > x \text{ with } p \mid \gcd(F(n),G(n)) \right] \\
= \sum_{n \leq x} \;&\mathbb{P}_{x}\!\left[\exists p > x \text{ with } p \mid F(n) \right] \cdot \mathbb{P}_{x}\!\left[\exists p > x \text{ with } p \mid G(n) \right] \gg x,
\end{align*}
as claimed.

\bibliographystyle{amsplain}

\end{document}